\documentclass[12pt]{elsarticle}
\usepackage{amssymb,amsmath,amsthm}

\def\cpoly{{\rm Cut}^\square(G)}
\def\cpolytope{{\rm Cut}^\square}

\newtheorem{Theorem}{Theorem}[section]
\newtheorem{Lemma}[Theorem]{Lemma}

\newtheorem{Proposition}[Theorem]{Proposition}

\theoremstyle{definition}
\newtheorem{Remark}[Theorem]{Remark}

\newtheorem{Example}[Theorem]{Example}

\newtheorem{Conjecture}[Theorem]{Conjecture}

\def\RR{{\mathbb R}}
\def\ZZ{{\mathbb Z}}

\def\NN{{\mathbb N}}

\def\PP{{\mathcal P}}

\journal{arXiv}

\begin{document}

\begin{frontmatter}

\title{Gorenstein cut polytopes}
\author{Hidefumi Ohsugi}

\address{
Department of Mathematics,
College of Science,
Rikkyo University,
Toshima-ku, Tokyo 171-8501, Japan} 
\ead{ohsugi@rikkyo.ac.jp}

\date{}

\begin{abstract}
An integral convex polytope ${\mathcal P}$ is said to be Gorenstein if 
its toric ring $K[{\mathcal P}]$ is normal and Gorenstein.
In this paper, Gorenstein cut polytopes of graphs 
are characterized explicitly.
First, we prove that Gorenstein cut polytopes are compressed
(i.e., all of whose reverse lexicographic triangulations are unimodular).
Second, by applying Athanasiadis's theory for
Gorenstein compressed polytopes,
we show that a cut polytope of a graph $G$ is Gorenstein
if and only if $G$ has no $K_5$-minor and $G$ is either
a bipartite graph without induced cycles of length $\geq 6$ or 
a bridgeless chordal graph.
\end{abstract}


\end{frontmatter}

\section*{Introduction}

Let $\PP \subset \RR^m$ be an integral convex polytope
of dimension $m$.
Let $K$ be a field and let
$K[{\bf x},{\bf x}^{- 1},t] = K[x_1,x_1^{-1},\ldots,x_m,x_m^{- 1},t]$ denote
the Laurent polynomial ring in $m+1$ variables over $K$.
The {\em toric ring} of the polytope $\PP$ is the subalgebra
$K[\PP]$ of $K[{\bf x},{\bf x}^{- 1},t]$ generated by those monomials
${\bf x}^{\bf a} t = x_1^{a_1} \cdots x_m^{a_m} t$ such that
${\bf a} =(a_1,\ldots,a_m) \in \PP \cap \ZZ^m$.
The defining ideal $I_\PP$ of $K[\PP]$ is called the {\em toric ideal} of $\PP$.
We regard $K[\PP]$ as
a homogeneous algebra
by setting each $\deg {\bf x}^{\bf a} t = 1$.
Let $H(K[\PP], r) = \dim_K K[\PP]_r$
where
$K[\PP]_r $ is the vector space spanned by the monomials of 
$K[\PP]$ of degree $r$.
Then, $H(K[\PP], r)$ is called the {\em Hilbert function} of $K[\PP]$
and 
$F(K[\PP], \lambda)=\sum_{r=0}^\infty H(K[\PP], r) \lambda^r$
is called the {\em Hilbert series} of $K[\PP]$.
It is known that we have
$F(K[\PP], \lambda) = 
\frac{h_0 + h_1 \lambda + \cdots + h_s \lambda^s }{(1 - \lambda )^{m+1}},$
where each $h_i \in \ZZ$ with $h_s \neq 0$.
The sequence
$(h_0, h_1, \ldots,h_s)$ is called the {\em $h$-vector} of 
$K[\PP]$.
An integral convex polytope $\PP$ is said to be {\em normal}
if $K[\PP]$ is a normal semigroup ring.
If $\PP$ is normal, then $K[\PP]$ is Cohen--Macaulay.
It is known \cite[Theorem 13.11]{Stu} 
 that $\PP$ is normal if and only if
the Hilbert function $H(K[\PP], r)$ is equal to
the {\em normalized Ehrhart polynomial}
$
E(\PP,r) = | r \PP \cap {\mathcal L}_\PP |
$
where $r \PP = \{r \alpha \ | \ \alpha \in \PP\}$
and ${\mathcal L}_\PP$ is a sublattice of $\ZZ^m$ spanned by $\PP \cap \ZZ^m$.
If $\PP$ is normal, then
the $h$-vector of $K[\PP]$ is nonnegative (i.e., $h_i \geq 0$ for all $0 \leq i \leq s$).
Moreover, if $\PP$ is normal, then
$K[\PP]$ is Gorenstein if and only if 
$h$-vector of $K[\PP]$ is symmetric
(i.e., $h_i = h_{s-i}$ for all $0 \leq i \leq s$).
An integral convex polytope $\PP$ is said to be {\it Gorenstein} if $K[\PP]$ is 
normal and Gorenstein (See, e.g., \cite{BrGu}).

In this paper, Gorenstein cut polytopes of graphs 
are characterized explicitly.
Let $G =(V(G),E(G))$ be a finite (undirected) graph on the vertex set 
$V(G) = [n]=\{1,2,\ldots,n\}$
and the edge set $E(G)=\{e_1,\ldots,e_m\}$.
We assume that $G$ has no loops and no multiple edges.
Given $S \subset [n]$,
the {\em cut semimetric} on $G$ induced by $S$ is
the $0/1$ vector
$\delta_G(S) = ( d_{i j} \ | \ \{i,j\} \in E(G)) \in \RR^m$ where
$$
d_{i j} :=
\left\{
\begin{array}{cc}
 1 & \mbox{ if } | S \cap \{i,j\}|=1,\\
 0 & \mbox{ otherwise,}
\end{array}
\right.
$$
for each $\{i,j \} \in E(G)$.
In particular, we have $\delta_G(\emptyset) = (0,\ldots,0) \in \RR^{m}$.
The {\em cut polytope} $\cpoly$ of $G$ is the convex hull of 
$ \{\delta_G (S) \ | \ S \subset [n]\} \subset \RR^m$.
If $S' =[n] \setminus S$, then we have $\delta_G(S) = \delta_G(S')$.
It then follows that $\cpoly$ has $2^{n-1}$ vertices.
It is known that the dimension of $\cpoly$ equals to $m$.
As explained in \cite{DeLa}, cut polytopes are well-known and
important objects in
discrete mathematics (graph theory, combinatorial optimization, etc.).

Sturmfels--Sullivant \cite{StSu} studied the toric rings of cut polytopes
and their applications to algebraic statistics.
Especially, they showed that the clique sum of two graphs $G_1$ and $G_2$ 
(i.e., the graph obtained by identifying a common clique of $G_1$ and $G_2$) yields
a toric fiber product \cite{Sul2} of the toric ideals of cut polytopes
$\cpolytope (G_1) $ and $\cpolytope (G_2) $.
They also gave several interesting conjectures on
properties of the toric ideals of cut polytopes.
Inspired by them, several results are known on toric ideals and toric rings
of cut polytopes.
See, e.g., \cite{Eng, KNP, NaPe, OhsCut, Sul3}.
However, in \cite{StSu}, they said that
``We do not have a firm conjecture on the structure of those graphs whose cut ideal is Gorenstein."
One of the reasons for this difficulty is that
Gorensteiness is {\em not} preserved under
taking the clique sum of graphs.
It is known \cite{NaPe} that, if $G$ is a tree, then $\cpoly$ is Gorenstein.
However, no other result seems to be known for Gorenstein cut polytopes.

The content of this paper is as follows.
In Section 1, together with graph theoretical terminology,
we summarize known results on the toric ring
of a cut polytope.
In Section 2, by using Barahona--Mahjoub's formula \cite{BaMa}
on the facets of normal cut polytopes
and Sullivant's characterization \cite{Sul} of compressed cut polytopes,
we prove that Gorenstein cut polytopes are compressed
(i.e., all of whose reverse lexicographic triangulations are unimodular).
In Section 3, 
by using Athanasiadis's theory of special simplices for
Gorenstein compressed polytopes,
we prove that a cut polytope of a graph $G$ is Gorenstein
if and only if $G$ has no $K_5$-minor and $G$ is either
a bipartite graph without induced cycles of length $\geq 6$ or 
a bridgeless chordal graph (Theorem \ref{main}).
%
%
%
%

\section{Normal cut polytopes}

In this section, we summarize known results on 
(i) normal cut polytopes, (ii) compressed cut polytopes
and (iii) Gorenstein cut polytopes.

First, we introduce graph theoretical terminology.
Let $G$ be a graph with the vertex set $V(G) = [n] = \{1,2,\ldots,n\}$
and the edge set $E(G)$.
We assume that $G$ has no loops and no multiple edges.
A {\em cycle} of length $q$ $(\geq 3)$ of $G$ is 
a finite sequence of the form
\begin{equation}
\label{defofcycle}
\Gamma =
 (
\{v_1,v_2\},
\{v_2,v_3\},
\ldots,
\{v_{q},v_1\}
)
\end{equation}
with each $\{v_k,v_{k+1}\}$, $1 \leq k \leq q-1$ and $\{v_q,v_1\}$ belong to $E(G)$
and $v_i \neq v_j$ for all $1 \leq i < j \leq q$.
An {\em even} (resp.~{\em odd}) {\em cycle} is a cycle of even (resp.~odd) length.
A {\em triangle} is a cycle of length 3.
A {\em chord} of a cycle (\ref{defofcycle}) is an edge $e \in E(G)$ of the form
$e = \{v_i,v_j\}$ for some $1 \leq i < j \leq q$ with $e \notin C$.
An {\em induced} cycle of $G$ is a cycle having no chord.
A graph $G$ is said to be {\em chordal} if 
$G$ has no induced cycle of length $\geq 4$.
An edge $e$ of $G$ is called a {\em bridge} if there exists no cycle of $G$
containing $e$.
It is easy to see that a chordal graph $G$ has no bridges if and only if
$E(G)$ is the union of all triangles of $G$.
A graph $G$ is called a {\em bipartite graph} if
there exists a bipartition $V(G) = V_1 \cup V_2$
such that any edge of $G$ connects a vertex of $V_1$
and a vertex of $V_2$.
It is known that $G$ is bipartite if and only if $G$ has no odd cycle.
Let $e =\{i,j\} \in E(G)$ be an edge of $G$.
Then, the new graph $G \setminus e$ on the vertex set $V(G)$
and the edge set $E(G)\setminus\{e\}$
is called
the graph obtained from $G$ by {\em deleting} the edge $e$.
On the other hand, the new graph $G / e$ obtained by 
the procedure
(i)
Identify the vertices $i$ and $j$;
(ii)
Delete the multiple edges that may be created while (i);
is called the graph obtained from $G$ by {\em contracting} the edge $e$.
A graph $H$ is said to be a {\em minor} of $G$
if it can be obtained from $G$ by a sequence of deletions and/or contractions
of edges (and deletions of vertices).
Let $K_n$ denote the complete graph with $n$ vertices.
It is known that

\begin{Proposition}[\cite{StSu}]
\label{normalityK5}
Let $G$ be a graph.
If $\cpoly$ is normal, then $G$ has no $K_5$-minor.
\end{Proposition}

It is conjectured 
that the converse of Proposition \ref{normalityK5} holds in general.

\begin{Conjecture}[\cite{StSu}]
Let $G$ be a graph.
Then, the following conditions are equivalent.
\begin{enumerate}
\item[(i)]
$\cpoly$ is normal.
\item[(ii)]
$K[\cpoly]$ is Cohen--Macaulay.
\item[(iii)]
$G$ has no $K_5$-minor.
\end{enumerate}
\end{Conjecture}

Although the conjecture is still open, it is known that

\begin{Proposition}[\cite{OhsCut}]
\label{minorclosed}
Let $H$ be a minor of a graph $G$.
If $\cpoly$ is normal, then so is $\cpolytope (H)$.
\end{Proposition}

In order to prove Proposition \ref{minorclosed},
the following Barahona--Mahjoub's formula \cite{BaMa}
plays an important role.

\begin{Proposition}
\label{barahona}
Let $G$ be a graph without $K_5$-minor.
Then, $\cpoly $ is the solution set of the following linear inequalities:
\begin{eqnarray*}
0 \leq x_i \leq 1 \hspace{1.5cm}
& & (e_i \mbox{ does not belong to any triangle of } G.)\\
\sum_{e_i \in F} x_i - \sum_{e_j \in C \setminus F} x_j \leq |F| -1
& & \ \ 
\left(
\begin{array}{c}
C \mbox{ is an induced cycle of } G,\\
F \mbox{ is an odd subset of } C.
\end{array}
\right)
\end{eqnarray*}
Moreover, each of them defines a facet of $\cpoly$.
\end{Proposition}

Let $G_1 = (V_1, E_1)$ and 
$G_2 = (V_2, E_2)$ be graphs such that $V_1 \cap V_2$ is a clique of both graphs.
The new graph $G = G_1 \sharp G_2$ with the vertex set $V = V_1 \cup V_2$
and edge set $E=E_1 \cup E_2$ is called the {\em clique sum} of $G_1$ and $G_2$
along $V_1 \cap V_2$.
If the cardinality of $V_1 \cap V_2$ is $k+1$, this operation is called a {\em $k$-sum} of the graphs.
The normality of cut polytopes are preserved under $0$, $1$, $2$-sum.

\begin{Proposition}[\cite{OhsCut}]
\label{csum}
Let $G = G_1 \sharp G_2$ be a $0$, $1$ or $2$-sum of $G_1$ and $G_2$.
Then,  $\cpoly$ is normal if and only if
both $\cpolytope (G_1)$ and $\cpolytope (G_2)$ are normal.
\end{Proposition}

Recall that $\cpoly$ is said to be compressed
if all of whose reverse lexicographic triangulations are unimodular.
In general, a compressed polytope is normal.
See \cite[Chapter 8]{Stu} and \cite{Sul} for details.
We now introduce the characterization of 
compressed cut polytope (\cite{Sul} and \cite[Theorem 1.3]{StSu}).

\begin{Proposition}
\label{compressed}
Let $G$ be a graph.
Then, the following conditions are equivalent:
\begin{enumerate}
\item[{\rm (i)}]
$\cpoly$ is compressed;
\item[{\rm (ii)}]
$\cpoly$ has a reverse lexicographic unimodular triangulation;
\item[{\rm (iii)}]
$G$ has no $K_5$-minor and no induced cycle of length $\geq 5$.
\end{enumerate}
\end{Proposition}

By Propositions \ref{barahona} and \ref{compressed},
we have the following:

\begin{Proposition}
\label{facet}
Suppose that the cut polytope $\cpoly$ of a graph $G$ is compressed.
Then, $\cpoly $ is the solution set of the following linear inequalities:
$$
\begin{array}{cl}
0 \leq x_i \leq 1,
& \ (e_i \mbox{ does not belong to any triangle of } G),\\
x_i - x_j - x_k \leq 0,
& \ ( \{e_i,e_j,e_k\} \mbox{ is a triangle of } G),\\
x_i + x_j + x_k \leq 2,
& \ ( \{e_i,e_j,e_k\} \mbox{ is a triangle of } G),\\
0 \leq x_i + x_j + x_k -x_\ell \leq 2,
& \ 
\left(
\begin{array}{c}
\{e_i,e_j,e_k,e_\ell\} \mbox{ is an induced cycle}\\
\mbox{of } G \mbox{ of length } 4
\end{array}
\right).
\end{array}
$$
Moreover, each of them defines a facet of $\cpoly$.
\end{Proposition}

On the other hand, in general, Gorensteiness is {\em not} preserved under
taking 
(i) the clique sums of graphs;
(ii) an edge contraction;
(iii) an edge deletion;
(iv) an induced subgraph.
The next example shows this fact.

\begin{Example}
Let $G_{m,n}$ be a 0-sum (i.e., glued at a vertex) of 
complete graphs $K_m$ and $K_n$.
By Proposition \ref{compressed},
$\cpolytope (K_2)$, $\cpolytope (K_3)$, $\cpolytope (G_{2,3})$
and $\cpolytope (G_{3,3})$ are compressed and hence normal.
Moreover,
\begin{enumerate}
\item[(a)]
Both $\cpolytope (K_2)$ and $\cpolytope (K_3)$ are Gorenstein (simplices).
However, $\cpolytope (G_{2,3})$ is not Gorenstein
since the $h$-vector 
$(1,3)$ of its toric ring is not symmetric.
\item[(b)]
The cut polytope $\cpolytope (G_{3,3})$ is Gorenstein
since the $h$-vector 
$(1,9,9,1)$ of its toric ring is symmetric.
However, $\cpolytope (G_{2,3})$ is not Gorenstein.
Note that $G_{2,3}$ is an induced subgraph of $G_{3,3}$.
In addition, $G_{2,3}$ is obtained by an edge contraction from $G_{3,3}$.
\end{enumerate}
\end{Example}

It is known \cite{NaPe} that, if $G$ is a tree, then $\cpoly$ is Gorenstein.
However, no other result seems to be known for Gorenstein cut polytopes.

\section{Gorenstein cut polytopes are compressed}

In this section, we show that
any Gorenstein cut polytope is compressed.
For  the cut polytope $\PP = \cpoly \subset \RR^m$ of $G$,
let $r \PP = \{r \alpha \ | \ \alpha \in \PP\}$ and
let ${\mathcal L}_\PP$ be a sublattice of $\ZZ^m$ spanned by the vertices of $\PP$.
The relative interior of $\PP$ is denoted by ${\rm int} (\PP)$.
The {\em codegree} of $\PP$ is 
$
{\rm codeg} (\PP) =
\min (k \in \NN \ | \ 
{\rm int} (k \PP) \cap {\mathcal L}_\PP \neq \emptyset)
$.
Then, the following is known (see, e.g., \cite{BrGu}):

\begin{Proposition}
\label{criterion}
Let $\PP = \cpoly$ be the cut polytope of a graph $G$
with codegree $d$.
If $\PP$ is Gorenstein, then 
${\rm int}( d \PP) \cap {\mathcal L}_\PP = \{{\bf v}\}$
and any vector 
${\bf w} \in {\rm int}( (d+r) \PP) \cap {\mathcal L}_\PP$
satisfies ${\bf w} - {\bf v} \in r \PP \cap {\mathcal L}_\PP$.
\end{Proposition}

So, the following proposition \cite[p. 258]{Lau} will play an important role.

\begin{Proposition}
\label{za}
Let $\PP = \cpoly$ be the cut polytope of a graph $G$.
Then, 
$(x_1,\ldots, x_m) \in {\mathcal L}_{\PP}$
if and only if 
$\sum_{e_i \in C} x_i \equiv 0 \mod 2$
for each cycle $C$ of $G$.
\end{Proposition}

Using these propositions, we have the following.

\begin{Theorem}
\label{GorCom}
If the cut polytope $\cpoly$ of a graph $G$ is Gorenstein,
then $\cpoly$ is compressed.
\end{Theorem}

\begin{proof}
Suppose that the cut polytope $\PP = \cpoly$ of a graph $G$ is Gorenstein
and not compressed.
Since $\PP$ is normal, $G$ has no $K_5$-minor by Proposition \ref{normalityK5}.
Thus, $G$ has an induced cycle of length $\geq 5$ by Proposition \ref{compressed}.
Let $d={\rm codeg}(\PP)$ and 
${\rm int}(d \PP) \cap {\mathcal L}_\PP = \{{\bf v}=(v_1,\ldots,v_m)\}$.
If $x_i = 2$ for all $1 \leq i \leq m$, then $0 < x_i < 4$ and 
$$
4(|F|-1) - 
\sum_{e_i \in F} x_i 
+ \sum_{e_j \in C \setminus F}x_j
=
2 (|C| - 2)
>0
$$
for any induced cycle $C$ and its odd subset $F$.
By Propositions
\ref{barahona} and \ref{za}, we have
$
(2,\ldots, 2) \in {\rm int} (4 \PP) \cap {\mathcal L}_\PP 
$.
Thus, $2 \leq d \leq 4$ and $v_i \leq 2$ for all $i=1,2,\ldots,m$
by Proposition \ref{criterion}.

\bigskip

\noindent
{\bf Case 1.}
There exists an induced even cycle $C$ of length $\geq 6$.

Let $C=(e_{i_1},\ldots,e_{i_{2\ell}})$ on the vertex set $\{p_1,\ldots, p_{2 \ell}\}$.
Suppose that 
$v_{i_1} + \cdots + v_{i_{2 \ell -1}} -v_{i_{2 \ell}} < d(2 \ell -2) -2$.
Let $\gamma = c_1 + \cdots +c_{2\ell} - 2 {\bf e}_{i_{2\ell}}$
where $c_j = \delta_G(\{ p_j \})$ for $1 \leq j \leq 2 \ell$.
By Proposition \ref{za}, $\gamma$ belongs to $ {\mathcal L}_\PP$.
Since $ \gamma_{i_1} + \cdots + \gamma_{i_{2\ell-1}} - \gamma_{i_{2\ell}}= 2(2\ell-1) > 2(2\ell-2)$,
it follows that $ \gamma \notin 2 \PP $.
We will show that $\beta =(\beta_1,\ldots,\beta_m)={\bf v} + \gamma$
belongs to ${\rm int}((d+2) \PP) \cap {\mathcal L}_\PP$.
Since $0 \leq \gamma_i \leq 2$ holds for each $1 \leq i \leq m$, 
we have $0< \beta_i <d+2$.
Let $C'$ be an induced cycle
and $F$ an odd subset $F$ of $C'$.
Set
\begin{eqnarray*}
\alpha & = &
 (d+2)(|F|-1) - \sum_{e_i \in F} \beta_i + \sum_{e_j \in C' \setminus F} \beta_j\\
&=&
d (|F|-1) - \sum_{e_i \in F} v_i + \sum_{e_j \in C' \setminus F} v_j
+
2(|F|-1) - \sum_{e_i \in F} \gamma_i + \sum_{e_j \in C' \setminus F} \gamma_j. \ \ 
\end{eqnarray*}
If $(C',F) \neq (C,\{e_{i_1},\ldots,e_{i_{2\ell-1}}\})$, then
we have 
$d (|F|-1) - \sum_{e_i \in F} v_i + \sum_{e_j \in C' \setminus F} v_j \geq 2$
and
$2(|F|-1) - \sum_{e_i \in F} \gamma_i + \sum_{e_j \in C' \setminus F} \gamma_j 
> 2(|F|-1) - 2 |F| = -2$.
Thus, $\alpha >0$.
If $(C',F) = (C,\{e_{i_1},\ldots,e_{i_{2\ell-1}}\})$, then
we have 
$d (|F|-1) - \sum_{e_i \in F} v_i + \sum_{e_j \in C' \setminus F} v_j \geq 4$
and
$2(|F|-1) - \sum_{e_i \in F} \gamma_i + \sum_{e_j \in C' \setminus F} \gamma_j 
= -2$.
Thus, $\alpha >0$.
Therefore, 
$\beta =(\beta_1,\ldots,\beta_m)={\bf v} + \gamma$
belongs to ${\rm int}((d+2) \PP) \cap {\mathcal L}_\PP$.
By Proposition \ref{criterion}, this contradicts that $\PP$ is Gorenstein.
Hence,
$v_{i_1} + \cdots + v_{i_{2 \ell -1}} -v_{i_{2 \ell}} = d(2 \ell -2) -2$.
By symmetry, we have
$- 2v_{i_{j}} + \sum_{k=1}^{2\ell} v_{i_{k}}= d(2 \ell -2) -2$
for all $j =1,2,\ldots,2\ell $.
Adding all of them, it follows that $(2\ell -2) \sum_{k=1}^{2\ell} v_{i_{k}}= 2 \ell (d (2 \ell -2) -2)$.
Thus, we have
$$
v_{i_{j}}  = 
\frac{\ell (d (\ell -1) -1)}{\ell -1} - (d(\ell -1) -1)
=
\frac{d (\ell -1) -1}{\ell -1}
=
d-
\frac{1}{\ell -1}.
$$
Since $\ell \geq 3$, this is not an integer, a contradiction.
(Note that, if $\ell =2$, then $v_{i_{j}}  = d-1$.)

\bigskip

\noindent
{\bf Case 2.}
There exists no induced even cycle of length $\geq 6$
and
there exists an induced odd cycle $C$ of length $\geq 5$.

Let $C=(e_{i_1},\ldots,e_{i_{2\ell+1}})$ on the vertex set $\{p_1,\ldots, p_{2 \ell+1}\}$.
Suppose that 
$v_{i_1} + \cdots + v_{i_{2 \ell+1}} < 2 \ell d-2$.
Let $\gamma = c_1 + \cdots +c_{2\ell+1} \in  {\mathcal L}_\PP$
where $c_j = \delta_G(\{ p_j \})$ for $1 \leq j \leq 2 \ell+1$.
Since $ \gamma_{i_1} + \cdots + \gamma_{i_{2\ell+1}}= 2(2\ell+1) > 2 \cdot 2\ell$,
it follows that $ \gamma \notin 2 \PP $.
We will show that $\beta =(\beta_1,\ldots,\beta_m)={\bf v} + \gamma$
belongs to ${\rm int}((d+2) \PP) \cap {\mathcal L}_\PP$.
Since $0 \leq \gamma_i \leq 2$ holds for each $1 \leq i \leq m$, 
we have $0< \beta_i <d+2$.
Let $C'$ be an induced cycle
and $F$ an odd subset $F$ of $C'$.
Set
\begin{eqnarray*}
\alpha & = &
 (d+2)(|F|-1) - \sum_{e_i \in F} \beta_i + \sum_{e_j \in C' \setminus F} \beta_j\\
&=&
d (|F|-1) - \sum_{e_i \in F} v_i + \sum_{e_j \in C' \setminus F} v_j
+
2(|F|-1) - \sum_{e_i \in F} \gamma_i + \sum_{e_j \in C' \setminus F} \gamma_j. \ \ 
\end{eqnarray*}
If $(C',F) \neq (C,C)$, then
we have 
$d (|F|-1) - \sum_{e_i \in F} v_i + \sum_{e_j \in C' \setminus F} v_j \geq 2$
and
$2(|F|-1) - \sum_{e_i \in F} \gamma_i + \sum_{e_j \in C' \setminus F} \gamma_j 
> 2(|F|-1) - 2 |F| = -2$.
Thus, $\alpha >0$.
If $(C',F) =(C,C)$, then
we have 
$d (|F|-1) - \sum_{e_i \in F} v_i + \sum_{e_j \in C' \setminus F} v_j \geq 4$
and
$2(|F|-1) - \sum_{e_i \in F} \gamma_i + \sum_{e_j \in C' \setminus F} \gamma_j 
= -2$.
Thus, $\alpha >0$.
Therefore, 
$\beta =(\beta_1,\ldots,\beta_m)={\bf v} + \gamma$
belongs to ${\rm int}((d+2) \PP) \cap {\mathcal L}_\PP$.
By Proposition \ref{criterion}, this contradicts that $\PP$ is Gorenstein.
Hence,
$v_{i_1} + \cdots +v_{i_{2 \ell+1}} = 2 \ell d -2$.

Since $v_i \leq 2$ for all $i$, we have
$2 \ell d -2 \leq 2 (2\ell +1)$ and hence
$d \leq 2+ 2/\ell \leq 3$.
If $d=2$, then $v_{i_1} + \cdots +v_{i_{2 \ell+1}} = 4 \ell -2$.
Since $\ell \geq 2$, we have $ 4 \ell -2 > 2 \ell +1$.
Hence, there exists $1 \leq k \leq 2 \ell +1$ such that $v_{i_k} = 2$.
By assumption, ${\bf v} \in {\rm int}(2 \PP)$.
If $e_{i_k}$ does not belong to any triangle of $G$,
then ${\bf v}$ belongs to the facet of $2\PP$ defined by $x_{i_k} = 2$, a contradiction.
Thus, there exists a triangle $(e_{i_k},e_{s}, e_{t})$ of $G$.
Since ${\bf v}$ belongs to ${\rm int}(2 \PP)$,
we have $2 + v_s + v_t < 4$ and $2 - v_s - v_t < 0$,
a contradiction.
Therefore, $d =3$.
Then, $3 \leq 2+ 2/\ell$ and hence $\ell = 2$.
In particular, $G$ has no induced odd cycle of length $\geq 7$.

Since  $v_{i_1} + \cdots +v_{i_5} = 10$, we have $v_{i_1}=v_{i_2}=v_{i_3}=v_{i_4}=v_{i_5}=2$.
Suppose that
there exists a triangle $(e_{i_1},e_{s}, e_{t})$ of $G$.
Since ${\bf v}$ belongs to ${\rm int}(3 \PP)$,
we have $2 + v_s + v_t < 6$ and $2 - v_s - v_t < 0$.
Thus, $2 < v_s + v_t < 4$ and hence $v_s + v_t =3$.
This contradicts that $2 + v_s + v_t$ is even.
Therefore, if an induced cycle contains $e_{i_1}$, 
then the length of the cycle is either 4 or 5.

Let $\gamma' = 2 e_{i_1} \in  {\mathcal L}_\PP \setminus 2\PP$.
We will show that $\beta' =(\beta_1',\ldots,\beta_m')={\bf v} + \gamma'$
belongs to ${\rm int}(5 \PP) \cap {\mathcal L}_\PP$.
Since $0 \leq \gamma_i \leq 2$ holds for each $1 \leq i \leq m$, 
we have $0< \beta_i' <5$.
Let $C'$ be an induced cycle
and $F$ an odd subset $F$ of $C'$.
Set
\begin{eqnarray*}
\alpha' & = &
 5 (|F|-1) - \sum_{e_i \in F} \beta_i' + \sum_{e_j \in C' \setminus F} \beta_j'\\
&=&
3 (|F|-1) - \sum_{e_i \in F} v_i + \sum_{e_j \in C' \setminus F} v_j
+
2(|F|-1) - \sum_{e_i \in F} \gamma_i' + \sum_{e_j \in C' \setminus F} \gamma_j'. \ \ 
\end{eqnarray*}
If $e_{i_1} \notin F$, then 
$2(|F|-1) - \sum_{e_i \in F} \gamma_i' + \sum_{e_j \in C' \setminus F} \gamma_j'
\geq 2(|F|-1) \geq 0$
and hence $\alpha' >0$.
If $|F|\geq 3$, then 
$2(|F|-1) - \sum_{e_i \in F} \gamma_i' + \sum_{e_j \in C' \setminus F} \gamma_j'
\geq 2(|F|-1) -2 \geq 2$
and hence $\alpha' >0$.
Suppose that $F = \{ e_{i_1} \}$.
Then, the length of $C'$ is 4 or 5.
If 
the length of $C'$ is 4, then $v_i =d-1= 2$ for all $i \in C'$
by the last argument in Case 1.
If the length of $C'$ is 5, then we have shown that
$v_i = 2$ for all $i \in C'$.
Thus, we have 
$\alpha' = -2 + 2(|C'|-1) -2= 2(|C'|-3) >0$.
Therefore, $\beta' =(\beta_1',\ldots,\beta_m')={\bf v} + \gamma'$
belongs to ${\rm int}(5 \PP) \cap {\mathcal L}_\PP$.
By Proposition \ref{criterion}, this contradicts that $\PP$ is Gorenstein.
\end{proof}


\begin{Remark}
Let $G$ be a graph.
Suppose that $\PP= \cpoly$ is normal.
By considering the vectors $(1,\ldots,1)$ and $(2,\ldots,2)$,
it is not difficult to see that
$$
{\rm codeg}(\PP)=
\left\{
\begin{array}{cc}
2 & \mbox{if } G \mbox{ is bipartite,}\\ 
4 & \mbox{if } G \mbox{ has a triangle,}\\
3 & \mbox{otherwise.}\\
\end{array}
\right.
$$
\end{Remark}

\section{Gorenstein cut polytopes}

In this section, in order to characterize Gorenstein cut polytopes,
we apply the theory of special simplices (given by Athanasiadis)
to compressed cut polytopes.
Let $\PP \subset \RR^m$ be a convex polytope.
A $d$-simplex $\Sigma$ each of whose vertices
is a vertex of $\PP$ is called a {\em special simplex} in $\PP$
if each facet of $\PP$
contains exactly $d$ of the vertices of $\Sigma$.
It is known (\cite{Ath, OHhvector}) that

\begin{Proposition}
\label{special}
Let $\PP$ be a compressed polytope. 
Then, $\PP$ is Gorenstein
if and only if 
there exists a special simplex in $\PP$.
Moreover, if $\PP$ is Gorenstein,
then the (symmetric) $h$-vector $(h_0,h_1,\ldots,h_s)$ of $K[\PP]$ 
satisfies $h_0 \leq h_1 \leq \cdots \leq h_{\lfloor s/2 \rfloor}$
(i.e., unimodal).
\end{Proposition}

We discuss the existence of special simplices of 
compressed cut polytopes.

\begin{Lemma}
\label{line}
Suppose that $\cpoly$ is compressed and possesses
a special simplex $\Sigma$ of $\cpoly$.
If $G$ satisfies one of the conditions
\begin{enumerate}
\item[{\rm (a)}]
There exists an edge $e_i$ of $G$ such that
no triangle of $G$ contains $e_i$,
\item[{\rm (b)}]
$G$ has an induced cycle of length $4$,
\end{enumerate}
then we have $\dim \Sigma = 1$.
\end{Lemma}

\begin{proof}
(a)
Suppose that there exists an edge $e_i$ of $G$ such that
no triangle of $G$ contains $e_i$.
By Proposition \ref{facet}, the inequalities $0 \leq x_i$ and $x_i \leq 1$
define the facets ${\mathcal F}$ and ${\mathcal F}'$, respectively, of $\cpoly$. 
Note that each vertex of $\cpoly$ belongs to exactly one of
${\mathcal F}$ and ${\mathcal F}'$.
Thus, the special simplex $\Sigma$ has exactly two vertices.

(b)
Suppose that $G$ has an induced cycle of length 4.
By Proposition \ref{facet}, the inequalities
$0 \leq x_i + x_j + x_k - x_\ell$ 
and
$x_i + x_j + x_k - x_\ell \leq 2$ 
define the facets ${\mathcal F}$ and ${\mathcal F}'$, respectively, of $\cpoly$. 
Since each vertex of $\cpoly$ belongs to exactly one of
${\mathcal F}$ and ${\mathcal F}'$,
the special simplex $\Sigma$ has exactly two vertices.
\end{proof}

\begin{Lemma}
\label{dim3}
Suppose that $\cpoly$ is compressed and possesses
a special simplex $\Sigma$ of $\cpoly$.
If $G$ has a triangle, then we have $\dim \Sigma = 3$.
\end{Lemma}

\begin{proof}
Let $C=\{e_i,e_j,e_k\}$ be a triangle of $G$.
By Proposition \ref{facet},
there are 4 facets of $\cpoly$ arising from $C$:
$$
\begin{array}{c}
{\mathcal F}_1 : x_i - x_j - x_k \leq 0, \ \ \ \ 
{\mathcal F}_2 : x_j - x_i - x_k \leq 0, \\
{\mathcal F}_3 : x_k - x_i - x_j \leq 0, \ \ \ \ 
{\mathcal F}_4 : x_i + x_j + x_k \leq 2.
\end{array}
$$
Given a vertex $(x_1,\ldots, x_m) \in \{0,1\}^m$ of $\cpoly$,
we have that the vector $(x_i,x_j,x_k)$ is one of the following:
$(0,0,0)$, $(0,1,1)$, $(1,0,1)$ and $(0,1,1)$.
Moreover, it follows that
\begin{eqnarray*}
(x_1,\ldots, x_m) \notin {\mathcal F}_1 & \Leftrightarrow & (x_i,x_j,x_k) = (0,1,1),\\
(x_1,\ldots, x_m) \notin {\mathcal F}_2 & \Leftrightarrow & (x_i,x_j,x_k) = (1,0,1),\\
(x_1,\ldots, x_m) \notin {\mathcal F}_3 & \Leftrightarrow & (x_i,x_j,x_k) = (1,1,0),\\
(x_1,\ldots, x_m) \notin {\mathcal F}_4 & \Leftrightarrow & (x_i,x_j,x_k) = (0,0,0).
\end{eqnarray*}
Thus, the special simplex $\Sigma$ has exactly four vertices, as required.
\end{proof}

We now come to the main result of the present paper.

\begin{Theorem}
\label{main}
Let $G$ be a graph.
Then, the cut polytope $\cpoly$ of $G$ is Gorenstein
(i.e., $K[\cpoly]$ is normal and Gorenstein)
if and only if 
$G$ has no $K_5$-minor and satisfies one of the following:
\begin{enumerate}
\item[{\em (i)}]
$G$ is a bipartite graph without induced cycle of length $\geq 6$;
\item[{\em (ii)}]
$G$ is a bridgeless chordal graph.
\end{enumerate}
\end{Theorem}

\begin{proof}
($\Rightarrow $)
Suppose that $\cpoly$ is Gorenstein.
By Theorem \ref{GorCom}, $\cpoly$ is compressed.
Hence, $G$ has no $K_5$-minor and no induced cycle of length $\geq 5$.
Since $\cpoly$ is compressed and Gorenstein,
$\cpoly$ has a special simplex $\Sigma$ by Proposition \ref{special}.

If $G$ is bipartite, then $G$ has no induced cycle of length $\geq 6$. 
If $G$ is not bipartite, then $G$ has an odd cycle.
Since $G$ has no induced odd cycle of length $\geq 5$,
$G$ has a triangle.
By Lemma \ref{dim3}, we have $\dim \Sigma =3$. 
Since $\dim \Sigma \neq 1$, 
$G$ satisfies neither (a) nor (b) in Lemma \ref{line}.
Thus,
$G$ is chordal and, for each $e \in E(G)$,
there exists a triangle $C$ of $G$ such that $e \in C$.

($\Leftarrow $)
Suppose that $G$ has no $K_5$-minor and 
$G$ is either a bipartite graph without induced cycle of length $\geq 6$
or a bridgeless chordal graph.
By Proposition \ref{compressed}, $\cpoly$ is compressed
and hence normal.
Thus, it is sufficient to show that $\cpoly$ has a special simplex.

Suppose that $G$ is bipartite.
Let $V(G) = V_1 \cup V_2$ be a bipartition of $V(G)$
and let $E(G) = \{e_1,\ldots, e_m\}$.
We will show that the simplex $\Sigma = {\rm Conv} ({\bf 0},{\bf 1})$,
where ${\bf 0} = (0,\ldots,0)$ and ${\bf 1} = (1,\ldots,1)$ is a special simplex.
Note that, both $\delta_G (\emptyset) = {\bf 0}$ and
$\delta_G (V_1) = {\bf 1}$ are vertices of $\cpoly$.
By Proposition \ref{facet},
the facets of $\cpoly$ are defined by the following inequalities:
$$
\begin{array}{cc}
0 \leq x_i \leq 1, & (1 \leq i \leq m),\\
0 \leq x_i + x_j + x_k - x_\ell \leq 2, & (\{e_i,e_j,e_k,e_\ell\} \mbox{ is a cycle of } G).
\end{array}
$$
For each facet ${\mathcal F}$ of $\cpoly$,
exactly one of ${\bf 0}$ and ${\bf 1}$ belongs to ${\mathcal F}$.
Thus, $\Sigma$ is a special simplex and hence $\cpoly$ is Gorenstein.

On the other hand, suppose that $G$ is chordal and, for each $e \in E(G)$,
there exists a triangle $C$ of $G$ such that $e \in C$.
Since $G$ is chordal and has no $K_5$ as a subgraph,
$G$ is four-colorable.
(It is known \cite[Proposition 5.5.2]{Diestel} that
any chordal graph is ``perfect."
See also Proposition \ref{another}.)
Let $V(G) = V_1 \cup V_2 \cup V_3 \cup V_4$ be a four coloring of $G$.
We define the vertex ${\bf u}_i$
of $\cpoly$ by 
${\bf u}_i = \delta_G (V_i)$ for each $i = 1,2,3, 4$.
We will show that $\Sigma = {\rm Conv}({\bf u}_1, {\bf u}_2, {\bf u}_3, {\bf u}_4)$
is a special simplex of $\cpoly$.
By Proposition \ref{facet},
the facets of $\cpoly$ are defined by the following inequalities:
$$
\begin{array}{cc}
x_i - x_j - x_k \leq 0, & (\{e_i,e_j,e_k\} \mbox{ is a triangle of } G),\\
x_i + x_j + x_k \leq 2, & (\{e_i,e_j,e_k\} \mbox{ is a triangle of } G).
\end{array}
$$
Let $C=\{e_i,e_j,e_k\}$ be a triangle of $G$ where 
$e_{i} = \{s,t\}$,
$e_{j} = \{t,u\}$ and
$e_{k} = \{s,u\}$.
There are 4 facets arising from $C$:
$$
\begin{array}{c}
{\mathcal F}_1 : x_i - x_j - x_k \leq 0, \ \ \ \ 
{\mathcal F}_2 : x_j - x_i - x_k \leq 0, \\
{\mathcal F}_3 : x_k - x_i - x_j \leq 0, \ \ \ \ 
{\mathcal F}_4 : x_i + x_j + x_k \leq 2.
\end{array}
$$
Without loss of generality,
we may assume that $s \in V_2$, $t \in V_3$ and $u \in V_4$.
Then, we have
$$
\begin{array}{c}
{\bf u}_1, {\bf u}_2, {\bf u}_3 \in {\mathcal F}_1,
 {\bf u}_4 \notin {\mathcal F}_1,\ \ \ \ 
{\bf u}_1, {\bf u}_3, {\bf u}_4 \in {\mathcal F}_2,
 {\bf u}_2 \notin {\mathcal F}_2,\\
{\bf u}_1, {\bf u}_2, {\bf u}_4 \in {\mathcal F}_3,
 {\bf u}_3 \notin {\mathcal F}_3,\ \ \ \ 
{\bf u}_2, {\bf u}_3, {\bf u}_4 \in {\mathcal F}_4,
 {\bf u}_1 \notin {\mathcal F}_4.
\end{array}
$$
This discussion is independent of the choice of a triangle of $G$.
Thus, $\Sigma$ is a special simplex and hence $\cpoly$ is Gorenstein.
\end{proof}

We give another characterization of a non-bipartite graph $G$
such that $\cpoly$ is Gorenstein:

\begin{Proposition}
\label{another}
Let $G$ be a connected graph.
Then, $G$ is a bridgeless chordal graph without $K_5$-minor
if and only if $G$ is a $0$, $1$ and $2$-sum of $K_3$'s and $K_4$'s.
\end{Proposition}

\begin{proof}
It is known \cite[Proposition 5.5.1]{Diestel} that
a graph $G$ is chordal if and only if
$G$ is a clique sum of complete graphs.

($\Rightarrow $)
Suppose that $G$ is a bridgeless chordal graph without $K_5$-minor.
Since $G$ is chordal, $G$ is a clique sum of complete graphs $K_{i_1}, \ldots , K_{i_r}$
($2 \leq i_1, \ldots, i_r \in \ZZ$).
Since $G$ has no $K_5$-minor,  we have $i_j \leq 4$ for all $1\leq j \leq r$.
Suppose that $i_j = 2$ for some $j$.
Then, $K_{i_j} = K_2$ consists of one edge $e$.
By the assumption, $e$ is not a bridge and hence $e$ belongs to other complete graph.
Hence, $K_{i_j}$ is redundant in the decomposition of $G$.
Thus, $G$ is a $0$, $1$ and $2$-sum of $K_3$'s and $K_4$'s.

($\Leftarrow $)
Suppose that $G$ is a $0$, $1$ and $2$-sum of $K_3$'s and $K_4$'s.
By  \cite[Proposition 5.5.1]{Diestel}, $G$ is a chordal graph. 
Since both $K_3$ and $K_4$ have no bridge, so does $G$.
Wagner's famous theorem says that, $G$ has no $K_5$-minor.
\end{proof}

\section*{Acknowledgement}

The author is grateful to Professor Takayuki Hibi
and anonymous referees
for their helpful comments.
This research was supported by the JST CREST.

\end{document}